\documentclass[11pt]{article}
\usepackage{amssymb, amsfonts, amsthm, mathrsfs, amsmath}
\usepackage[all,cmtip]{xy}
\usepackage{tipa}

\theoremstyle{definition}
\newtheorem{thm}{Theorem}[section]
\newtheorem{dfn}[thm]{Definition}
\newtheorem{lemma}[thm]{Lemma}
\newtheorem{cor}[thm]{Corollary}
\newtheorem{prop}[thm]{Proposition}

\newtheorem{rmk}[thm]{Remark}

\def\endo{\mathrm{End}}

\def\id{\mathrm{id}}

\def\rr{\mathbb{R}}

\def\clf{\mathcal{F}}

\def\clj{\mathcal{J}}
\def\clk{\mathcal{K}}

\def\clo{\mathcal{O}}
\def\clp{\mathcal{P}}
\def\clq{\mathcal{Q}}
\def\clr{\mathcal{R}}

\def\frg{\mathfrak{g}}

\def\pt{\partial}
\def\p{\partial}

\def\ud{\mathrm{d}}

\def\spann{\mathrm{Span}}

\def\bea{\begin{eqnarray}}
\def\eea{\end{eqnarray}}

\def\<{{\langle}}
\def\>{{\rangle}}

\begin{document}

\title{$\clf$-harmonic 3-forms and the Type IIA flow on 6-dimensional symplectic Lie algebras}
\author{Teng Fei\footnote{Work supported by NSF LEAPS-MPS Grant 2418918 and Simons Collaboration Grant 830816.}}

\date{}

\maketitle{}

\begin{abstract}
In this paper, we investigate the problem of $\clf$-harmonic forms and the long-time behavior of the Type IIA flow on certain nilpotent and solvable symplectic Lie algebra (and corresponding nilmanifolds and solvmanifolds). In particular, we demonstrate that the Type IIA flow in these cases are very useful to detect desired geometric structures including Lagrangian torus fibrations and harmonic almost complex structures.
\end{abstract}

\section{Introduction}

Given a 6-dimensional symplectic manifold $(M,\omega)$, from any 3-form $\varphi$ on $M$, one can canonically construct an endomorphism $\clk(\varphi)$ of $TM$, another 3-form $\clf(\varphi)$, and a scalar function $\clq(\varphi)$ \cite{fei2025}, where $\clk$, $\clf$, and $\clq$ are homogeneous polynomials in $\varphi$ of degree 2, 3, and 4 respectively. Many classical work in complex geometry and mathematical physics can be incorporated into this framework, such as \cite{hitchin2000, tomasiello2008, tseng2014, donaldson2024, donaldson2024b, donaldson2024c}.

In this paper, we would like to address two analytic problems related to the aforementioned construction, namely the problem of $\clf$-harmonic forms and the problem of solving the Type IIA flow.

As a nonlinear analogue of the standard Hodge theory, we say a 3-form $\varphi$ on $(M,\omega)$ is \emph{$\clf$-harmonic} if both $\ud\varphi=0$ and $\ud \clf(\varphi)=0$. A basic question is that how do we characterize the set of $\clf$-harmonic 3-forms. In particular, we would like to know if an $\clf$-harmonic representative always exist if its cohomology is fixed. Since we may work with a symplectic but non-K\"ahler manifold, it is necessary to consider both the de Rham cohomology and the symplectic cohomology \cite{tseng2012a, tseng2012b, tsai2016}.  

In a series of papers \cite{fei2021b, fei2023, fei2021g}, we introduce and investigate the Type IIA flow as an intrinsically defined geometric flow on symplectic 6-manifolds. In the language of \cite{fei2025}, the Type IIA flow on a fixed symplectic manifold $(M,\omega)$ is the following evolution equation of 3-forms
\bea
\pt_t\varphi=\ud\Lambda_\omega\ud \clf(\varphi),\label{iiaflow}
\eea
where $\Lambda_\omega$ is the Lefschetz operator of contraction with respect to the symplectic form $\omega$. Conventionally we require the initial data $\varphi\big|_{t=0}=\varphi_0$ to be a $\ud$-closed 3-form that belongs to the orbit $\clo_-^+(\mu)$ pointwise so we have the desired geometric interpretation as well as the short-time existence and uniqueness. However in principle, Equation (\ref{iiaflow}) makes sense for any initial data though additional effort is in need for its short-time existence. In this paper, we would like to explore the relationship between the long time behavior of (\ref{iiaflow}) and the underlying structure of $(M,\omega)$, in the context of arbitrary initial data. In particular, it encompasses the characterization of $\clf$-harmonic forms, as they constitute a special kind of stationary points of (\ref{iiaflow}).

In general, understanding the set of $\clf$-harmonic forms and the behavior of the Type IIA flow (\ref{iiaflow}) requires development of new powerful techniques in nonlinear analysis. As a preliminary attempt on these two problems, in this paper we would like to avoid such analytical difficulties by imposing a large symmetry on $(M,\omega)$. To be more specific, we would like to assume that $(M,\omega)$ is a 6-dimensional symplectic Lie group (or its quotient by discrete subgroups) with everything being left-invariant. To this effect, we can replace the standard de Rham and symplectic cohomology theories by the corresponding Lie algebra cohomology and its symplectic analogue. As a consequence, we can reduce the problem of $\clf$-harmonic forms to a collection of algebraic equations and the equation of Type IIA flow to an ODE system. 

For simplicity, we shall consider only those 6-dimensional symplectic Lie groups with cocompact discrete subgroups. It is well-known (see for example \cite{milnor1976}) that any Lie group admitting a cocompact discrete subgroup must be unimodular. By a theorem of Chu \cite{chu1974}, any unimodular symplectic Lie group must be solvable. Though the classification is not fully complete, the structure of 6-dimensional symplectic solvable Lie groups is well-understood from the works \cite{goze1987, goze1996, salamon2001, khakimdjanov2004, campoamor-stursberg2009, baues2016}. In particular, it contains many more examples compared to 6-dimensional Lie groups with left-invariant symplectic half-flat structures \cite{conti2007, fernandez2013}. In this paper, we do not attempt to exhaust all the possible cases. Instead, we present a more refined analysis of the two examples given in \cite[Section 9.3.2]{fei2021b}. Further cases will be left for future studies.

This paper is organized as follows. In Section 2, we review necessary backgrounds about 3-forms over a 6-dimensional space and various cohomology theories on a symplectic Lie algebra. In Section 3, we establish a number of formulae for calculation purpose. Section 4 is devoted to address the problem of $\clf$-harmonic forms and the long-time behavior of the Type IIA flow. We present one nilpotent example and one solvable example. From these explicit calculations we show that the case of nilpotent Lie algebra is drastically different from the case of solvable Lie algebra. Moreover we also learn that in these examples the Type IIA flow can be very useful in detecting desired geometric structures such as Lagrangian fibrations and harmonic almost complex structures.

\section{Backgrounds}

Let $G$ be a Lie group of real dimension $n$. All the left-invariant data on $G$ are determined by their value at the identity, hence they can be viewed as an element of the Lie algebra $\frg$ or other algebraic constructions built from $\frg$. After fixing a basis $\{e_1,e_2,\dots,e_n\}$ of $\frg$, the isomorphism class of $\frg$ is determined by the structure constants $\{c^k_{ij}\}_{1\leq i,j,k\leq n}$, which are defined by $[e_i,e_j]=c^k_{ij}e_k$. In many cases, it is more convenient to work with the dual Lie algebra $\frg^*$ and its exterior product 
\[\bigwedge\nolimits^*\frg^*=\bigoplus_{i=0}^n\bigwedge\nolimits^i\frg^*.\]
The structure of $\frg$ can be recovered from the exterior derivative $\ud$ defined on $\frg^*$, given by
\[\ud e^k=-\frac{1}{2}c^k_{ij}e^i\wedge e^j,\]
where $\{e^1,\dots,e^n\}$ and $c^k_{ij}$ are the dual basis and the structure constants associated to the basis $\{e_1,e_2,\dots,e_n\}$. Extending the above defined $\ud$ by linearity and the graded Leibniz rule, one gets a differential $\ud:\bigwedge^i\frg^*\to\bigwedge^{i+1}\frg^*$ satisfying $\ud^2=0$. The cohomology associated to the cochain complex $(\bigwedge^*\frg^*,\ud)$ is known as the \emph{Lie algebra cohomology} of $\frg$ \cite{chevalley1948}. In many cases, the Lie algebra cohomology of $\frg$ coincides with the de Rham cohomology of quotients of $G$ by its cocompact discrete subgroups, including the case when $G$ is compact or nilpotent \cite{nomizu1954}.

For the convenience of notation, we shall use a symbol like $(0,0,0,e^{15},0,e^{13})$ to denote the 6-dimensional Lie algebra $\frg$ defined by a basis $\{e^1,\dots,e^6\}$ of $\frg^*$ satisfying
\[\ud e^1=\ud e^2=\ud e^3=\ud e^5=0,\quad\ud e^4=e^1\wedge e^5,\quad\ud e^6=e^1\wedge e^3.\] 

We say $(G,\omega)$ is a \emph{symplectic Lie group} or $(\frg,\omega)$ is a \emph{symplectic Lie algebra} if $\omega\in\bigwedge^2\frg^*$ is a $\ud$-closed and non-degenerate 2-form on $\frg$. For a symplectic Lie group or a symplectic Lie algebra, its dimension must an even number $2m$. For any $0\leq k\leq m$, we can define the space of primitive $k$-forms as 
\[\clp^k\frg^*=\left\{\alpha\in\bigwedge\nolimits^k\frg^*:\omega^{m-k+1}\wedge\alpha=0\right\}.\]
We have the well-known Lefschetz decomposition
\[\bigwedge\nolimits^k\frg^*=\clp^k\frg^*\oplus\left(\omega\wedge\clp^{k-2}\frg^*\right)\oplus\left(\omega^2\wedge\clp^{k-4}\frg^*\right)\oplus\dots\]
and the Lefschetz isomorphism
\[\omega^{m-k}:\bigwedge\nolimits^k\frg^*\cong\bigwedge\nolimits^{2m-k}\frg^*.\]

Following Tseng-Yau \cite{tseng2012a, tseng2012b}, one can easily show that
\[\ud:\clp^k\frg^*\to\clp^{k+1}\frg^*\oplus\left(\omega\wedge\clp^{k-1}\frg^*\right),\]
therefore one decompose $\ud=\p_++\omega\wedge\p_-$, where $\p_\pm:\clp^k\frg^*\to\clp^{k\pm1}\frg^*$ are first order differential operators mapping primitive forms to primitive forms. The \emph{symplectic cohomology} $\textrm{SH}^*$of $(\frg,\omega)$ is defined to be the cohomology of the cochain complex
\[0\to\clp^0\frg^*\xrightarrow{\p_+}\clp^1\frg^*\xrightarrow{\p_+}\dots\xrightarrow{\p_+}\clp^m\frg^*\xrightarrow{\p_+\p_-}\clp^m\frg^*\xrightarrow{\p_-} \clp^{m-1}\frg^*\xrightarrow{\p_-}\dots\xrightarrow{\p_-}\clp^0\frg^*\to0.\]

In this paper, we are concerned with three kinds of cohomologies on $\frg$. They are the de Rham/Lie algebra cohomology group
\[\textrm{H}^m(\frg;\rr):=\frac{\ker\left(\ud:\bigwedge^m\frg^*\to\bigwedge^{m+1}\frg^*\right)}{\textrm{im} \left(\ud:\bigwedge^{m-1}\frg^*\to\bigwedge^{m}\frg^*\right)},\]
the primitive part of the de Rham/Lie algebra cohomology
\[\textrm{PH}^m(\frg;\rr):=\frac{\clp^m\frg^*\cap\ker\left(\ud:\bigwedge^m\frg^*\to\bigwedge^{m+1}\frg^*\right)}{\clp^m\frg^*\cap\textrm{im} \left(\ud:\bigwedge^{m-1}\frg^*\to\bigwedge^{m}\frg^*\right)},\]
and the symplectic cohomologies
\[\textrm{SH}_+^m(\frg;\rr):=\frac{\ker\left(\p_+\p_-:\clp^m\frg^*\to\clp^m\frg^*\right)}{\textrm{im}\left(\p_+:\clp^{m-1}\frg^*\to\clp^m\frg^*\right)}\]
and
\[\textrm{SH}_-^m(\frg;\rr):=\frac{\ker\left(\p_-:\clp^m\frg^*\to\clp^{m-1}\frg^*\right)}{\textrm{im}\left(\p_+\p_-:\clp^m\frg^*\to\clp^m\frg^*\right)}.\]
The Poincar\'e duality says the natural pairing $\wedge:\textrm{SH}_+^m(\frg;\rr)\otimes\textrm{SH}_-^m(\frg;\rr)\to\bigwedge^{2m}\frg^*$ is non-degenerate if $\frg$ is unimodular.\\
Since $\p_+\p_-=\ud\Lambda_\omega\ud$ and $\ker\p_-=\ker\ud\cap\clp^m\frg^*$, we naturally have the surjective map
\[\textrm{SH}_-^m(\frg;\rr)\twoheadrightarrow\textrm{PH}^m(\frg;\rr)\]
and the injective map
\[\textrm{PH}^m(\frg;\rr)\hookrightarrow\textrm{H}^m(\frg;\rr).\]
On the other hand, as $\textrm{im}~\ud\cap\clp^m\frg^*\subset\textrm{im}~\p_+$, we also have a natural map
\[\textrm{SH}_+^m(\frg;\rr)\to\textrm{PH}^m(\frg;\rr).\]
 
\section{Computational Formulae}

In this section, we establish some basic computational formulae for 3-forms in real dimension 6. We first review the basic constructions of the equivariant homogeneous polynomials $\clk$, $\clf$, and $\clq$ following \cite{fei2025}. Some other references include \cite{hitchin2000} and \cite{fei2015b}.

Let $V$ be a 6-dimensional real vector space equipped with a symplectic form $\omega$. For any 3-form $\varphi\in\bigwedge^3V^*$, we can define $\clk(\varphi)\in\endo~ V\otimes\bigwedge^6V^*$ by
\[\clk(\varphi)(v)=-\iota_v\varphi\wedge\varphi\in\bigwedge\nolimits^5V^*\cong V\otimes\bigwedge\nolimits^6V^*\]
for any $v\in V$. The isomorphism $\bigwedge^5V^*\cong V\otimes\bigwedge^6V^*$ is canonical. If we choose a basis $\{e_1,\dots,e_6\}$ of $V$ with dual basis $\{e^1,\dots,e^6\}$ of $V^*$, we have
\[\clk(\varphi)(v)=-\sum_{i=1}^6e_i\otimes e^i\wedge\iota_v\varphi\wedge\varphi\in V\otimes\bigwedge\nolimits^6V^*.\]
In addition, we can define $\clf(\varphi)\in\bigwedge^3V^*\otimes\bigwedge^6V^*$ by
\[\clf(\varphi)(v_1,v_2,v_3)=-2\varphi(\clk(\varphi)(v_1),v_2,v_3)\]
and $\clq(\varphi)\in(\bigwedge^6V^*)^{\otimes2}$ by
\[\clq(\varphi)=-\varphi\wedge\clf(\varphi).\]
It is clear from their definitions that $\clk$, $\clf$, and $\clq$ are quadratic, cubic, and quartic in $\varphi$ respectively. Moreover, $\clk$, $\clf$, and $\clq$ satisfies the following identities:
\begin{prop}\cite[Proposition 2.5]{fei2025}~
\begin{enumerate}
\item $\clk(\varphi)\circ\clk(\varphi)=\dfrac{\id_V}{4}\cdot\clq(\varphi)\in\endo~V\otimes(\bigwedge^6V^*)^{\otimes2}$,
\item $\clk(\clf(\varphi))=-\clk(\varphi)\cdot\clq(\varphi)\in\endo~V\otimes(\bigwedge^6V^*)^{\otimes3}$,
\item $\clf(\clf(\varphi))=-\varphi\cdot\clq^2(\varphi)\in\bigwedge^3V^*\otimes(\bigwedge^6V^*)^{\otimes4}$.
\end{enumerate}
\end{prop}
What we presented so far is independent of the symplectic structure. The symplectic form $\omega$ determines a canonical volume form $\omega^3/3!\in\bigwedge^6V^*$, with which we can view $\clk(\varphi)$, $\clf(\varphi)$ and $\clq(\varphi)$ as elements in $\endo~V$, $\bigwedge^3V^*$ and $\rr$ respectively.

A notable example is that for $\varphi$ such that $\clq(\varphi)<0$, we can define a complex structure $\clj(\varphi)$ on $V$ by
\[\clj(\varphi)=\frac{2\clk(\varphi)}{\sqrt{-\clq(\varphi)}}.\]
The 3-form $\varphi$ is the real part of a complex $(3,0)$-form $\varphi+i\hat\varphi$ with respect to $\clj(\varphi)$. The imaginary part $\hat\varphi$ is purely determined from $\varphi$ by
\[\hat\varphi=\clj(\varphi)^*\varphi=\frac{\clf(\varphi)}{\sqrt{-\clq(\varphi)}}.\]
If we define $|\varphi|^2$ by
\[\varphi\wedge\clf(\varphi)=|\varphi|^4\frac{\omega^3}{3!},\]
then we have $\clk(\varphi)=\dfrac{|\varphi|^2}{2}\clj(\varphi)$, $\clf(\varphi)=|\varphi|^2\hat\varphi$, and $\clq(\varphi)=-|\varphi|^4$.

Under the presence of the symplectic form $\omega$, we usually only consider $\varphi$ that is primitive with respect to $\omega$. In this case, $\clf(\varphi)$ is also primitive. If in addition that $\clq(\varphi)<0$, we know that $\omega$ is a $(1,1)$-form with respect to the complex structure $\clj(\varphi)$ and that $|\varphi|^2$ is indeed the norm square of $\varphi$ with respect to the metric (not necessarily positive definite) $g(\cdot,\cdot)=\omega(\cdot,\clj(\varphi)\cdot)$.

Given a 6-dimensional symplectic vector space $(V,\omega)$, we can choose a basis $\{e_1,\dots,e_6\}$ of $V$ and its dual basis $\{e^1,\dots,e^6\}$ of $V^*$ such that the symplectic form $\omega$ takes the standard form $\omega=e^{12}+e^{34}+e^{56}$. Under such a choice, a general 3-form $\varphi\in\bigwedge^3V^*$ takes the form
\bea
\varphi&=&Ae^{135}+Be^{136}+Ce^{145}+De^{146}+Ee^{235}+Fe^{236}+Ge^{245}+He^{246}\label{generalphi}\\
&&+(Ie^1+Je^2)(e^{34}-e^{56})+(Ke^3+Le^4)(e^{12}-e^{56})+(Me^5+Ne^6)(e^{12}-e^{34})\nonumber\\
&&+(Oe^1+Pe^2)(e^{34}+e^{56})+(Qe^3+Re^4)(e^{12}+e^{56})+(Se^5+Te^6)(e^{12}+e^{34}),\nonumber
\eea
where $A,B,\dots,S,T$ are constants. And $\varphi$ is primitive if and only if the last 6 coefficients $O,P,Q,R,S,T$ are all zero, namely $\varphi$ takes the form
\bea
\varphi&=&Ae^{135}+Be^{136}+Ce^{145}+De^{146}+Ee^{235}+Fe^{236}+Ge^{245}+He^{246}\label{primitivephi}\\
&&+(Ie^1+Je^2)(e^{34}-e^{56})+(Ke^3+Le^4)(e^{12}-e^{56})+(Me^5+Ne^6)(e^{12}-e^{34}).\nonumber
\eea

For later use, we shall compute the expression of $\clk(\varphi)$, $\clf(\varphi)$, and $\clq(\varphi)$ for $\varphi$ in (\ref{generalphi}).
\begin{lemma}\label{k}
The endomorphism $\clk(\varphi)$ is given by
\bea
\clk(\varphi)(e_1)&=&(AH-BG-CF+DE+2IJ-2KR+2LQ-2MT+2NS-2OP)e_1\nonumber\\
&&-2(AD-BC+I^2-O^2)e_2\nonumber\\
&&-2(C(N+T)-D(M+S)-(I-O)(L+R))e_3\nonumber\\
&&+2(A(N+T)-B(M+S)-(I-O)(K+Q))e_4\nonumber\\
&&+2(B(L+R)-D(K+Q)-(I+O)(N+T))e_5\nonumber\\
&&-2(A(L+R)-C(K+Q)-(I+O)(M+S))e_6.\nonumber
\eea
\bea
\clk(\varphi)(e_2)&=&2(EH-FG+J^2-P^2)e_1\nonumber\\
&&-(AH-BG-CF+DE+2IJ+2KR-2LQ+2MT-2NS-2OP)e_2\nonumber\\
&&-2(G(N+T)-H(M+S)-(J-P)(L+R))e_3\nonumber\\
&&+2(E(N+T)-F(M+S)-(J-P)(K+Q))e_4\nonumber\\
&&+2(F(L+R)-H(K+Q)-(J+P)(N+T))e_5\nonumber\\
&&-2(E(L+R)-G(K+Q)-(J+P)(M+S))e_6.\nonumber
\eea
\bea
\clk(\varphi)(e_3)&=&-2(E(N-T)-F(M-S)-(J+P)(K-Q))e_1\nonumber\\
&&+2(A(N-T)-B(M-S)-(I+O)(K-Q))e_2\nonumber\\
&&+(AH-BG+CF-DE-2IP+2JO+2KL+2MT-2NS-2QR)e_3\nonumber\\
&&-2(AF-BE+K^2-Q^2)e_4\nonumber\\
&&-2(B(J+P)-F(I+O)-(K+Q)(N-T))e_5\nonumber\\
&&+2(A(J+P)-E(I+O)-(K+Q)(M-S))e_6.\nonumber
\eea
\bea
\clk(\varphi)(e_4)&=&-2(G(N-T)-H(M-S)-(J+P)(L-R))e_1\nonumber\\
&&+2(C(N-T)-D(M-S)-(I+O)(L-R))e_2\nonumber\\
&&+2(CH-DG+L^2-R^2)e_3\nonumber\\
&&-(AH-BG+CF-DE+2IP-2JO+2KL-2MT+2NS-2QR)e_4\nonumber\\
&&-2(D(J+P)-H(I+O)-(L+R)(N-T))e_5\nonumber\\
&&+2(C(J+P)-G(I+O)-(L+R)(M-S))e_6.\nonumber
\eea
\bea
\clk(\varphi)(e_5)&=&2(E(L-R)-G(K-Q)-(J-P)(M-S))e_1\nonumber\\
&&-2(A(L-R)-C(K-Q)-(I-O)(M-S))e_2\nonumber\\
&&-2(C(J-P)-G(I-O)-(L-R)(M+S))e_3\nonumber\\
&&+2(A(J-P)-E(I-O)-(K-Q)(M+S))e_4\nonumber\\
&&+(AH+BG-CF-DE+2IP-2JO+2KR-2LQ+2MN-2ST)e_5\nonumber\\
&&-2(AG-CE+M^2-S^2)e_6.\nonumber
\eea
\bea
\clk(\varphi)(e_6)&=&2(F(L-R)-H(K-Q)-(J-P)(N-T))e_1\nonumber\\
&&-2(B(L-R)-D(K-Q)-(I-O)(N-T))e_2\nonumber\\
&&-2(D(J-P)-H(I-O)-(L-R)(N+T))e_3\nonumber\\
&&+2(B(J-P)-F(I-O)-(K-Q)(N+T))e_4\nonumber\\
&&+2(BH-DF+N^2-T^2)e_5\nonumber\\
&&-(AH+BG-CF-DE-2IP+2JO-2KR+2LQ+2MN-2ST)e_6.\nonumber
\eea
\end{lemma}
\begin{lemma}\label{f}
The expression of $\clf(\varphi)$ satisfies
\bea
-\frac{1}{2}\clf(\varphi)&=&\hat Ae^{135}+\hat Be^{136}+\hat Ce^{145}+\hat De^{146}+\hat Ee^{235}+\hat Fe^{236}+\hat Ge^{245}+\hat He^{246}\label{generalf}\\
&&+(\hat Ie^1+\hat Je^2)(e^{34}-e^{56})+(\hat Ke^3+\hat Le^4)(e^{12}-e^{56})+(\hat Me^5+\hat Ne^6)(e^{12}-e^{34})\nonumber\\
&&+(\hat Oe^1+\hat Pe^2)(e^{34}+e^{56})+(\hat Qe^3+\hat Re^4)(e^{12}+e^{56})+(\hat Se^5+\hat Te^6)(e^{12}+e^{34}),\nonumber
\eea
where
\bea
\hat A&=&A(AH-BG-CF-DE+2IJ+2KL+2MN-2OP-2QR-2ST)\nonumber\\
&&-2(B(M^2-S^2)+C(K^2-Q^2)+E(I^2-O^2)-BCE)\nonumber\\
&&-4(IKM+OKS-IQS-OQM),\nonumber
\eea
\bea
\hat B&=&B(AH-BG+CF+DE+2IJ+2KL-2MN-2OP-2QR+2ST)\nonumber\\
&&-2(-A(N^2-T^2)+D(K^2-Q^2)+F(I^2-O^2)+ADF)\nonumber\\
&&-4(IKN+OKT-IQT-OQN),\nonumber
\eea
\bea
\hat C&=&C(AH+BG-CF+DE+2IJ-2KL+2MN-2OP+2QR-2ST)\nonumber\\
&&-2(-A(L^2-R^2)+D(M^2-S^2)+G(I^2-O^2)+ADG)\nonumber\\
&&-4(ILM+OLS-IRS-ORM),\nonumber
\eea
\bea
\hat D&=&D(-AH-BG-CF+DE+2IJ-2KL-2MN-2OP+2QR+2ST)\nonumber\\
&&-2(-B(L^2-R^2)-C(N^2-T^2)+H(I^2-O^2)-BCH)\nonumber\\
&&-4(ILN+OLT-IRT-ORN),\nonumber
\eea
\bea
\hat E&=&E(AH+BG+CF-DE-2IJ+2KL+2MN+2OP-2QR-2ST)\nonumber\\
&&-2(-A(J^2-P^2)+F(M^2-S^2)+G(K^2-Q^2)+AFG)\nonumber\\
&&-4(JKM+PKS-JQS-PQM),\nonumber
\eea
\bea
\hat F&=&F(-AH-BG+CF-DE-2IJ+2KL-2MN+2OP-2QR+2ST)\nonumber\\
&&-2(-B(J^2-P^2)+H(K^2-Q^2)-E(N^2-T^2)-BEH)\nonumber\\
&&-4(JKN+PKT-JQT-PQN),\nonumber
\eea
\bea
\hat G&=&G(-AH+BG-CF-DE-2IJ-2KL+2MN+2OP+2QR-2ST)\nonumber\\
&&-2(-C(J^2-P^2)-E(L^2-R^2)+H(M^2-S^2)-CEH)\nonumber\\
&&-4(JLM+PLS-JRS-PRM),\nonumber
\eea
\bea
\hat H&=&H(-AH+BG+CF+DE-2IJ-2KL-2MN+2OP+2QR+2ST)\nonumber\\
&&-2(-D(J^2-P^2)-F(L^2-R^2)-G(N^2-T^2)+DFG)\nonumber\\
&&-4(JLN+PLT-JRT-PRN),\nonumber
\eea
and
\bea
\hat I&=&I(AH-BG-CF+DE)-2J(AD-BC)\nonumber\\
&&+2(A(LN-RT)-B(LM-RS)-C(KN-QT)+D(KM-QS))\nonumber\\
&&-2O(IP-JO+KR-LQ-MT+NS),\nonumber
\eea
\bea
\hat J&=&J(-AH+BG+CF-DE)+2I(EH-FG)\nonumber\\
&&+2(E(LN-RT)-F(LM-RS)-G(KN-QT)+H(KM-QS))\nonumber\\
&&-2P(IP-JO+KR-LQ-MT+NS)\nonumber
\eea
\bea
\hat K&=&K(AH-BG+CF-DE)-2L(AF-BE)\nonumber\\
&&+2(A(JN+PT)-B(JM+PS)-E(IN+OT)+F(IM+OS))\nonumber\\
&&-2Q(IP-JO+KR-LQ+MT-NS),\nonumber
\eea
\bea
\hat L&=&L(-AH+BG-CF+DE)+2K(CH-DG)\nonumber\\
&&+2(C(JN+PT)-D(JM+PS)-G(IN+OT)+H(IM+OS))\nonumber\\
&&-2R(IP-JO+KR-LQ+MT-NS),\nonumber
\eea
\bea
\hat M&=&M(AH+BG-CF-DE)-2N(AG-CE)\nonumber\\
&&+2(A(JL-PR)-C(JK-PQ)-E(IL-OR)+G(IK-OQ))\nonumber\\
&&+2S(IP-JO-KR+LQ-MT+NS),\nonumber
\eea
\bea
\hat N&=&N(-AH-BG+CF+DE)+2M(BH-DF)\nonumber\\
&&+2(B(JL-PR)-D(JK-PQ)-F(IL-OR)+H(IK-OQ))\nonumber\\
&&+2T(IP-JO-KR+LQ-MT+NS),\nonumber
\eea
\bea
\hat O&=&O(AH-BG-CF+DE)-2P(AD-BC)\nonumber\\
&&-2(A(LT-RN)-B(LS-RM)-C(KT-QN)+D(KS-QM))\nonumber\\
&&-2I(IP-JO+KR-LQ-MT+NS)\nonumber
\eea
\bea
\hat P&=&P(-AH+BG+CF-DE)+2O(EH-FG)\nonumber\\
&&-2(E(LT-RN)-F(LS-RM)-G(KT-QN)+H(KS-QM))\nonumber\\
&&-2J(IP-JO+KR-LQ-MT+NS)\nonumber
\eea
\bea
\hat Q&=&Q(AH-BG+CF-DE)-2R(AF-BE)\nonumber\\
&&+2(A(JT+PN)-B(JS+PM)-E(IT+ON)+F(IS+OM))\nonumber\\
&&-2K(IP-JO+KR-LQ+MT-NS),\nonumber
\eea
\bea
\hat R&=&R(-AH+BG-CF+DE)+2Q(CH-DG)\nonumber\\
&&+2(C(JT+PN)-D(JS+PM)-G(IT+ON)+H(IS+OM))\nonumber\\
&&-2L(IP-JO+KR-LQ+MT-NS),\nonumber
\eea
\bea
\hat S&=&S(AH+BG-CF-DE)-2T(AG-CE)\nonumber\\
&&+2(A(JR-PL)-C(JQ-PK)-E(IR-OL)+G(IQ-OK))\nonumber\\
&&+2M(IP-JO-KR+LQ-MT+NS),\nonumber
\eea
\bea
\hat T&=&T(-AH-BG+CF+DE)+2S(BH-DF)\nonumber\\
&&+2(B(JR-PL)-D(JQ-PK)-F(IR-OL)+H(IQ-OK))\nonumber\\
&&+2N(IP-JO-KR+LQ-MT+NS).\nonumber
\eea
\end{lemma}
\begin{lemma}\label{q}
We have
\begin{equation}
\begin{split}
&\frac{\clq(\varphi)}{4}\label{generalq}\\
=&2(A^2H^2+B^2G^2+C^2F^2+D^2E^2)-(AH+BG+CF+DE)^2+4(ADFG+BCEH)\\
+&4(I^2-O^2)(FG-EH)+4(J^2-P^2)(BC-AD)+4(IJ-OP)(AH-BG-CF+DE)\\
+&4(K^2-Q^2)(DG-CH)+4(L^2-R^2)(BE-AF)+4(KL-QR)(AH-BG+CF-DE)\\
+&4(M^2-S^2)(DF-BH)+4(N^2-T^2)(CE-AG)+4(MN-ST)(AH+BG-CF-DE)\\
+&8A(JLN-JRT+PLT-PRN)-8B(JLM-JRS+PLS-PRM)\\
-&8C(JKN-JQT+PKT-PQN)+8D(JKM-JQS+PKS-PQM)\\
-&8E(ILN-IRT+OLT-ORN)+8F(ILM-IRS+OLS-ORM)\\
+&8G(IKN-IQT+OKT-OQN)-8H(IKM-IQS+OKS-OQM)\\
+&8\left[(IP-JO)^2+(LQ-KR)^2+(MT-NS)^2\right]-4(IP-JO-KR+LQ+MT-NS)^2.\nonumber
\end{split}
\end{equation}
\end{lemma}
\begin{proof}
Lemma \ref{k}, \ref{f} and \ref{q} can be proved by brute force calculation.
\end{proof}
\begin{rmk}~\\
In many cases, we only need Lemma \ref{k}, \ref{f} and \ref{q} for $\varphi$ being primitive. In these cases, we put $O=P=Q=R=S=T=0$ and obtain simpler formulae.
\end{rmk}

In \cite{hitchin2000}, Hitchin essentially proved that
\[\delta\clq(\varphi)=-4\frac{\delta\varphi\wedge\clf(\varphi)}{\omega^3/3!},\]
which holds for arbitrary $\varphi$ by its algebraicity. Using the notation in Lemma \ref{q}, we have
\[\begin{split}\frac{\clq(\varphi)}{2}=&-A\hat H+B\hat G+C\hat F-D\hat E+E\hat D-F\hat C-G\hat B+H\hat A-2I\hat J+2J\hat I-2K\hat L+2L\hat K\\
&-2M\hat N+2M\hat N+2O\hat P-2P\hat O+2Q\hat R-2R\hat Q+2S\hat T-2T\hat S.\end{split}\]
Hitchin's result in local coordinates is equivalent to the following seeming-wrongly-scaled formulae:
\bea
\frac{\pt\clq}{\pt A}=-8\hat H,&\quad &\frac{\pt\clq}{\pt H}=8\hat A,\nonumber\\
\frac{\pt\clq}{\pt B}=8\hat G,&\quad &\frac{\pt\clq}{\pt G}=-8\hat B,\nonumber\\
\frac{\pt\clq}{\pt C}=8\hat F,&\quad &\frac{\pt\clq}{\pt F}=-8\hat C,\nonumber\\
\frac{\pt\clq}{\pt D}=-8\hat E,&\quad &\frac{\pt\clq}{\pt E}=8\hat D,\nonumber
\eea
and
\bea
\frac{\pt\clq}{\pt I}=-16\hat J,&\quad &\frac{\pt\clq}{\pt J}=16\hat I,\nonumber\\
\frac{\pt\clq}{\pt K}=-16\hat L,&\quad &\frac{\pt\clq}{\pt L}=16\hat K,\nonumber\\
\frac{\pt\clq}{\pt M}=-16\hat N,&\quad &\frac{\pt\clq}{\pt N}=16\hat M.\nonumber\\
\frac{\pt\clq}{\pt O}=16\hat P,&\quad &\frac{\pt\clq}{\pt P}=-16\hat O.\nonumber\\
\frac{\pt\clq}{\pt Q}=16\hat R,&\quad &\frac{\pt\clq}{\pt R}=-16\hat Q.\nonumber\\
\frac{\pt\clq}{\pt S}=16\hat T,&\quad &\frac{\pt\clq}{\pt T}=-16\hat S.\nonumber
\eea
In fact, the scaling here is a consequence of the Euler's identity.

\section{Examples}

In this section, we study the problem of existence and uniqueness of $\clf$-harmonic forms, and the long-time behavior of the Type IIA flow on symplectic Lie groups with left invariant data. To be more specific, we would like to address the following two questions.\\

\noindent\textbf{Question 1}: The existence and uniqueness of $\clf$-harmonic 3-forms in a fixed cohomology class.\\

By definition, a 3-form $\varphi$ is $\clf$-harmonic if $\ud\varphi=0$ and $\ud\clf(\varphi)=0$. As an analogue of the standard Hodge theory, we would like to address Question 1 in the set-up of a symplectic Lie algebra $(\frg,\omega)$ with $\varphi\in\bigwedge^3\frg^*$. Since $\ud\varphi=0$, we know that $\varphi$ defines a de Rham cohomology class in $\textrm{H}^3(\frg;\rr)$ for a general 3-form, and that $\varphi$ defines a cohomology class in $\textrm{PH}^3(\frg;\rr)$ and $\textrm{SH}^3_\pm(\frg;\rr)$ if $\varphi$ is primitive. We would like to know that, if the cohomology class of $\varphi$ is fixed in the above sense (de Rham for general $\varphi$, and one of the three kinds of cohomologies for primitive $\varphi$), does there exist a unique $\clf$-harmonic representative.

For a given symplectic Lie algebra $(\frg,\omega)$, to answer Question 1, we proceed along the following steps.

\begin{description}
\item{\textbf{Step 1}:} Compute $\textrm{H}^3(\frg;\rr)$, $\textrm{PH}^3(\frg;\rr)$, and $\textrm{SH}^3_\pm(\frg;\rr)$ explicitly.
\item{\textbf{Step 2}:} By assuming $\varphi$ taking the form of (\ref{generalphi}) or (\ref{primitivephi}), we reduce the $F$-harmonic equation to a system of polynomial equations.
\item{\textbf{Step 3}:} Determine the existence and uniqueness of solutions to the system of algebraic equations in Step 2 within each cohomology class.
\end{description}

\noindent\textbf{Question 2}: The long-time behavior of the Type IIA flow.\\

We would like to investigate the long-time behavior of the Type IIA flow (\ref{iiaflow}) on symplectic Lie groups and their quotients, including whether the flow has long-time existence of finite-time singularities. In its original setup \cite{fei2021b}, the initial data is both closed, primitive, and positive. However, our formulation in (\ref{iiaflow}) allows us to extend the initial data to be a general 3-form, not necessarily closed or primitive. In this paper, we will consider such a flow with initial data under three kinds of generalities: (1) a general 3-form; (2) a primitive 3-form; (3) a closed primitive 3-form.

It is worth mentioning that there is little difference between the case of general initial data and that of primitive initial data. The Lefschetz decomposition for 3-form reads 
\[\bigwedge\nolimits^3\frg^*=\clp^3\frg^*\oplus\omega\wedge\frg^*.\]
For any 1-form $\alpha\in\frg^*$, we have that
\[\ud\Lambda_\omega\ud(\omega\wedge\alpha)=\ud\Lambda_\omega(\omega\wedge\ud\alpha)=\ud(\ud\alpha+\omega\wedge\Lambda_\omega\ud\alpha)=0\]
since $\Lambda_\omega\ud\alpha$ is a constant. This computation implies that $\ud\Lambda_\omega\ud\clf(\varphi)$ has only primitive components, hence the non-primitive components of $\varphi$ are stationary under the Type IIA flow (\ref{iiaflow}) with invariant data.

To address Question 2, we first reduce the Type IIA flow to an ODE system. Then, we find its stationary points. Finally, we analyze the long-time behavior of this ODE system.

\subsection{A nilpotent example}

Now let us take a closer look at the nilmanifold considered in \cite[Example 5.2]{bartolomeis2006} and \cite[pp. 798-799]{fei2021b}, where the underlying nilpotent Lie algebra $\frg$ is characterized by $(0,0,0,e^{15},0,e^{13})$ with the symplectic form $\omega$
\[\omega=e^{12}+e^{34}+e^{56}.\]

\noindent For Question 1, firstly, we compute the various relevant cohomology groups. It is straightforward to obtain the following:
\begin{prop}\label{prop1a}~\\
\begin{enumerate}
\item The space of $\ud$-closed 3-forms is spanned by the following basis vectors
\[\begin{split}&e^{135},e^{136},e^{145},e^{146},e^{235},e^{236},e^{245},e^1(e^{34}-e^{56}),e^3(e^{12}-e^{56}),e^5(e^{12}-e^{34}),\\ &e^1(e^{34}+e^{56}),e^2(e^{34}+e^{56}),e^3(e^{12}+e^{56}),e^5(e^{12}+e^{34}),e^{124},e^{126}.\end{split}\]
\item The space of $\ud$-closed primitive 3-forms is spanned by the following basis vectors
\[e^{135},e^{136},e^{145},e^{146},e^{235},e^{236},e^{245},e^1(e^{34}-e^{56}),e^3(e^{12}-e^{56}),e^5(e^{12}-e^{34}).\]
\item The space of $\ud$-exact 3-forms is spanned by the following basis vectors
\[e^{135}=\ud(e^{34})=-\ud(e^{56}),e^1(e^{34}-e^{56})=-\ud(e^{46}),e^{123}=\ud(e^{26}),e^{125}=\ud(e^{24}).\]
\item The space of $\ud$-exact primitive 3-forms is spanned by the following basis vectors
\[e^{135}=\ud(e^{34})=-\ud(e^{56}),e^1(e^{34}-e^{56})=-\ud(e^{46}).\]
\item The space of $\p_+\p_-$-closed primitive 3-forms is spanned by the following vectors
\[\begin{split}&e^{135},e^{136},e^{145},e^{146},e^{235},e^{236},e^{245},e^1(e^{34}-e^{56}),e^2(e^{34}-e^{56}),e^3(e^{12}-e^{56}),\\
&e^4(e^{12}-e^{56}),e^5(e^{12}-e^{34}),e^6(e^{12}-e^{34}).\end{split}\]
\item The space of $\p_+\p_-$-exact primitive 3-forms is spanned by the following vector
\[e^{135}=-\frac{1}{2}\p_+\p_-(e^{246}).\]
\item The space of $\p_+$-exact primitive 3-forms is spanned by the following vectors
\[\begin{split}&e^{135}=\frac{1}{2}\p_+(e^{34}-e^{56}),e^1(e^{34}-e^{56})=-\p_+(e^{46}),e^3(e^{12}-e^{56})=2\p_+(e^{26}),\\
&e^5(e^{12}-e^{34})=2\p_+(e^{24}).\end{split}\]
\end{enumerate}
\end{prop}
As a consequence, we conclude that
\begin{prop}\label{prop1b}~\\
\begin{enumerate}
\item $\textrm{H}^3(\frg;\rr)$ has a basis
\[e^{136},e^{145},e^{146},e^{235},e^{236},e^{245},e^2(e^{34}+e^{56}),e^{134}=e^{156},e^{124},e^{126},e^{356},e^{345}.\]
\item $\textrm{PH}^3(\frg;\rr)$ has a basis
\[e^{136},e^{145},e^{146},e^{235},e^{236},e^{245},e^3(e^{12}-e^{56}),e^5(e^{12}-e^{34}).\]
\item $\textrm{SH}_+^3(\frg;\rr)$ has a basis
\[e^{136},e^{145},e^{146},e^{235},e^{236},e^{245},e^2(e^{34}-e^{56}),e^4(e^{12}-e^{56}),e^6(e^{12}-e^{34}).\]
\item $\textrm{SH}_-^3(\frg;\rr)$ has a basis
\[e^{136},e^{145},e^{146},e^{235},e^{236},e^{245},e^1(e^{34}-e^{56}),e^3(e^{12}-e^{56}),e^5(e^{12}-e^{34}).\]
\end{enumerate}
\end{prop}

Secondly, from Proposition \ref{prop1a}(a), we know that for $\varphi$ taking the form of (\ref{generalphi}), it is $\ud$-closed if and only if
\[H=J=0,\quad L=R,\quad N=T.\]
Consequently, the 3-form $\varphi$ being $F$-harmonic is equivalent to the system
\[\begin{cases}&H=J=0,\quad L=R,\quad N=T,\\
&\hat H=\hat J=0,\quad \hat L=\hat R,\quad \hat N=\hat T.\end{cases}\]
By applying formulae in Lemma \ref{f}, the $F$-harmonicity of $\varphi$ is reduced to the following system of polynomial equations
\begin{equation}\label{alg1}
\begin{cases}&D(P^2+FG)=0,\\
&I(P^2+FG)+T[G(K-Q)-P(M-S)]+R[P(K-Q)+F(M-S)]=0,\\
&D[G(K-Q)-P(M-S)]=0,\\
&D[P(K-Q)+F(M-S)]=0.
\end{cases}
\end{equation}
For the next step, we notice that the parameters for $\textrm{H}^3(\frg;\rr)$ are $\{B,C,D,E,F,G,K-Q,M-S,O,P,R,T\}$ and the parameters in a fixed de Rham cohomology class are
$\{A,I,K+Q,M+S\}$, therefore there is no uniqueness result. The existence of $F$-harmonic form depends on the de Rham cohomology class. The loci of cohomology classes admitting an $F$-harmonic representative is the disjoint union of the following quasi-affine varieties in $\textrm{H}^3(\frg;\rr)$
\[\{D=0,P^2+FG\neq 0\}\]
and
\[\{D=P^2+FG=T[G(K-Q)-P(M-S)]+R[P(K-Q)+F(M-S)]=0\}\]
and
\[\left\{D\neq 0,\textrm{Rank}\begin{bmatrix} G & -P & M-S \\ P & F & K-Q \end{bmatrix}\leq 1\right\}.\]
 
In particular, there is an affine variety $V\subset\textrm{H}^3(\frg;\rr)\cong\rr^{12}$ defined by
\[\begin{split}V:=&\{D=P^2+FG=T[G(K-Q)-P(M-S)]+R[P(K-Q)+F(M-S)]=0\}\\
\coprod&\left\{\textrm{Rank}\begin{bmatrix} G & -P & M-S \\ P & F & K-Q \end{bmatrix}\leq 1\right\}\end{split}\]
such that for any cohomology class $[\varphi]\in V$, all representatives of $[\varphi]$ are $\clf$-harmonic.

Inspired by the above calculation, we introduce the following definitions.
\begin{dfn}~
\begin{enumerate}
\item The \emph{$\clf$-harmonic locus} $L_\clf$ of a 6-dimensional symplectic Lie algebra $(\frg,\omega)$ (with respect to the de Rham/Lie algebra cohomology) is the subset of $\textrm{H}^3(\frg;\rr)$ consisting of cohomology classes that admit an $\clf$-harmonic representative.
\item The \emph{perfect $\clf$-harmonic locus} $PL_\clf$ of a 6-dimensional symplectic Lie algebra $(\frg,\omega)$ (with respect to the de Rham/Lie algebra cohomology) is the subset of $\textrm{H}^3(\frg;\rr)$ consisting of cohomology classes such that all representatives are $\clf$-harmonic.
\end{enumerate}
\end{dfn}
\begin{rmk}
Clearly by definition we have that $PL_\clf\subset L_\clf\subset\textrm{H}^3(\frg;\rr)$. In addition, one can show that $L_\clf$ is always a quasi-affine variety, and $PL_\clf$ is always an affine variety. Furthermore, one can define similar notions for other cohomologies including $\textrm{PH}^3(\frg;\rr)$ and $\textrm{SH}_\pm^3(\frg;\rr)$. In this example, we can characterize $PL_\clf$ and $L_\clf$ by
\begin{enumerate}
\item $[\varphi]\in PL_\clf$ if and only if $\clq$ is a constant on $[\varphi]$.
\item $[\varphi]\in L_\clf\setminus PL_\clf$ if and only if $\clq$ achieves a strict extremal value in $[\varphi]$ at any of its $\clf$-harmonic representatives.
\end{enumerate}
\end{rmk}

When $\varphi$ is primitive, by setting $O=P=Q=R=S=T=0$ in (\ref{alg1}), we get a simplified system of polynomial equations
\begin{equation}
DFG=IFG=DFM=DGK=0
\end{equation}
with $H=J=L=N=0$. In this case
\[\frac{\clq(\varphi)}{4}=2(BG+CF+DE)^2-(BG+CF+DE)^2\geq 0\]
since $DFG=0$. 

Summarizing results in Proposition \ref{prop1a} and \ref{prop1b}, we obtain the following table\\

\begin{tabular}{c|c|c|c}
& cohomology param. & exact param. & $\clf$-harmonic locus\\
\hline
$\textrm{PH}^3(\frg;\rr)$ & $B,C,D,E,F,G,K,M$ & $A,I$ & $DFG=DFM=DGK=0$\\
\hline
$\textrm{SH}_+^3(\frg;\rr)$ & $B,C,D,E,F,G$ & $A,I,K,M$ & $DFG=0$\\
\hline
$\textrm{SH}_-^3(\frg;\rr)$ & $B\!,\!C,\!D\!,\!E,\!F\!,\!G\!,\!I\!,\!K\!,\!M$ & $A$ & $DFG\!=\!IFG\!=\!DFM\!=\!DGK\!=\!0$
\end{tabular}\\

In any of the three cohomologies $\textrm{PH}^3(\frg;\rr)$ or $\textrm{SH}^3_\pm(\frg;\rr)$, $A$ is always a parameter of the exact forms, therefore we can never have uniqueness in a fixed cohomology class. Meanwhile as a consequence of above calculation, we have the following curious fact.
\begin{cor}
If $[\varphi]\in\textrm{SH}^3_-(\frg;\rr)$ admits an $F$-harmonic representative, then every representative of $[\varphi]$ is $\clf$-harmonic.
\end{cor}

In fact, we can compute the perfect $\clf$-harmonic locus of $(\frg,\omega)$ with respect to each of the cohomology groups above, as summarized in the following table\\

\begin{tabular}{c|c|c|c}
& cohomology param. & exact param. & perfect $\clf$-harmonic locus\\
\hline
$\textrm{PH}^3(\frg;\rr)$ & $B,C,D,E,F,G,K,M$ & $A,I$ & $FG=DFM=DGK=0$\\
\hline
$\textrm{SH}_+^3(\frg;\rr)$ & $B,C,D,E,F,G$ & $A,I,K,M$ & $FG=DF=DG=0$\\
\hline
$\textrm{SH}_-^3(\frg;\rr)$ & $B\!,\!C,\!D\!,\!E,\!F\!,\!G\!,\!I\!,\!K\!,\!M$ & $A$ & $DFG\!=\!IFG\!=\!DFM\!=\!DGK\!=\!0$
\end{tabular}\\

\noindent For Question 2, we notice that $\textrm{im}~\ud\Lambda_\omega\ud$ is a 1-dimensional space spanned by $e^{135}$, whose coefficient in (\ref{generalphi}) is $A$. With notations from (\ref{generalphi}) and (\ref{f}), the Type IIA flow  reduces to an ODE
\begin{equation}\label{ode}
\p_t A=4\hat H=-4AH^2+\mathcal{R},
\end{equation}
where
\[\begin{split}&\mathcal{R}=4H(BG+CF+DE-2IJ-2KL-2MN+2OP+2QR-2ST)\\
&+8\left[D(J^2-P^2)+F(L^2-R^2)+G(N^2-T^2)-DFG-2JLN-2PLT+2JRT+2PRN\right],\end{split}\]
and $B,C,\dots,S,T$ are constants depending only on initial data.
The solution of (\ref{ode}) is 
\[A(t)=e^{-4H^2t}\left[\frac{\mathcal{R}(e^{4H^2t}-1)}{4H^2}+A(0)\right].\]
When $H=0$, the above expression should be understood as $A(t)=A(0)+\mathcal{R}t$. We see that the flow always has long time existence and it converges if and only if $H\neq 0$ or $H=\mathcal{R}=0$.

The divergent case is that $H=0$ and $\mathcal{R}\neq 0$, where we have
\[\lim_{t\to\infty}\frac{\varphi(t)}{\mathcal{R}t}=e^{135}=:\varphi_{\infty}.\] 
In this case $\ker\varphi_\infty=\spann\{e_2,e_4,e_6\}$ defines a Lagrangian foliation on the nilmanifold. In fact, it gives rise to a Lagrangian fibration since the Lagrangian subspace $\spann\{e_2,e_4,e_6\}$ of the Lie algebra is an ideal and it corresponds to a normal Lagrangian subgroup of the nilpotent Lie group. This verifies the metrical picture that along the Type IIA flow, the fibers of this Lagrangian fibration collapses and the manifold after scaling converges to its base in the sense of Gromov-Hausdorff. 

When $H\neq 0$, we have 
\[\lim_{t\to\infty}A(t)=\frac{\mathcal{R}}{4H^2}.\]
In other words, $\varphi(t)$ converges to a left-invariant 3-form with $\hat H=0$, which is a stationary points of the Type IIA flow.

\subsection{A solvable example}

Next, let us investigate the solvmanifold in \cite[Theorem 4.5]{tomassini2008} and \cite[pp. 799-802]{fei2021b} in greater generality, where the Lie algebra of the 6-dimensional solvable Lie group is 
\[(-\lambda e^{15},\lambda e^{25},-\lambda e^{36},\lambda e^{46},0,0)\]
with $\lambda=\log\dfrac{3+\sqrt{5}}{2}$ and the standard symplectic form $e^{12}+e^{34}+e^{56}$. By choosing a rescaled basis, it is equivalent to work with
\[(-e^{15},e^{25},-e^{36},e^{46},0,0)\]
with $\omega=\dfrac{1}{\lambda^2}(e^{12}+e^{34}+e^{56})$.\\
~\\
\noindent As before, it is straightforward to obtain the following:
\begin{prop}\label{prop2a}~\\
\begin{enumerate}
\item The space of $\ud$-closed 3-forms is spanned by the following basis vectors
\[\begin{split}&e^{135}+e^{136},e^{145}-e^{146},e^{235}-e^{236},e^{245}+e^{246},e^5(e^{12}-e^{34}),e^6(e^{12}-e^{34}),\\ &e^5(e^{12}+e^{34}),e^6(e^{12}+e^{34}),e^{156},e^{256},e^{356},e^{456}.\end{split}\]
\item The space of $\ud$-closed primitive 3-forms is spanned by the following basis vectors
\[e^{135}+e^{136},e^{145}-e^{146},e^{235}-e^{236},e^{245}+e^{246},e^5(e^{12}-e^{34}),e^6(e^{12}-e^{34}).\]
\item The space of $\ud$-exact 3-forms is spanned by the following basis vectors
\[\begin{split}&e^{135}+e^{136}=\ud(e^{13}),e^{145}-e^{146}=\ud(e^{14}),e^{235}-e^{236}=-\ud(e^{23}),e^{245}+e^{246}=-\ud(e^{24}),\\
&e^{156}=-\ud(e^{16}),e^{256}=\ud(e^{26}),e^{356}=\ud(e^{35}),e^{456}=-\ud(e^{45}).\end{split}\]
\item The space of $\ud$-exact primitive 3-forms is spanned by the following basis vectors
\[e^{135}+e^{136}=\ud(e^{13}),e^{145}-e^{146}=\ud(e^{14}),e^{235}-e^{236}=-\ud(e^{23}),e^{245}+e^{246}=-\ud(e^{24}).\]
\item The space of $\p_+\p_-$-closed primitive 3-forms is spanned by the following vectors
\[\begin{split}&e^{135}+e^{136},e^{145}-e^{146},e^{235}-e^{236},e^{245}+e^{246},e^1(e^{34}-e^{56}),e^2(e^{34}-e^{56}),e^3(e^{12}-e^{56}),\\
&e^4(e^{12}-e^{56}),e^5(e^{12}-e^{34}),e^6(e^{12}-e^{34}).\end{split}\]
\item The space of $\p_+\p_-$-exact primitive 3-forms is spanned by the following vectors
\[\begin{split}&e^{135}+e^{136}=-\frac{1}{\lambda^2}\p_+\p_-(e^{135})=\frac{1}{\lambda^2}\p_+\p_-(e^{136}),\\
&e^{145}-e^{146}=\frac{1}{\lambda^2}\p_+\p_-(e^{145})=\frac{1}{\lambda^2}\p_+\p_-(e^{146}),\\
&e^{235}-e^{236}=\frac{1}{\lambda^2}\p_+\p_-(e^{235})=\frac{1}{\lambda^2}\p_+\p_-(e^{236}),\\
&e^{245}+e^{246}=-\frac{1}{\lambda^2}\p_+\p_-(e^{245})=\frac{1}{\lambda^2}\p_+\p_-(e^{246}).\end{split}\]
\item The space of $\p_+$-exact primitive 3-forms is spanned by the following vectors
\[\begin{split}&e^{135}\!+\!e^{136}=\p_+(e^{13}),e^{145}\!-\!e^{146}=\p_+(e^{14}),e^{235}\!-\!e^{236}=-\p_+(e^{23}),e^{245}\!+\!e^{246}=-\p_+(e^{24}),\\
&e^1(e^{34}-e^{56})=2\p_+(e^{16}),e^2(e^{34}-e^{56})=-2\p_+(e^{26}),e^3(e^{12}-e^{56})=-2\p_+(e^{35}),\\
&e^4(e^{12}-e^{56})=2\p_+(e^{45}).\end{split}\]
\end{enumerate}
\end{prop}

As a consequence, a general closed 3-form $\varphi$ takes the form
\bea
\varphi&=&A(e^{135}\!+\!e^{136})\!+\!C(e^{145}\!-\!e^{146})\!+\!E(e^{235}\!-\!e^{236})\!+\!G(e^{245}\!+\!e^{246})\!+\!(Me^5\!+\!Ne^6)(e^{12}\!-\!e^{34})\nonumber\\
&&+(Se^5+Te^6)(e^{12}+e^{34})+2Oe^{156}+2Pe^{256}+2Qe^{356}+2Re^{456},\label{phi2}
\eea
meaning that in (\ref{generalphi}) we have
\begin{equation}
A-B=C+D=E+F=G-H=I+O=J+P=K+Q=L+R=0.\label{closed} 
\end{equation}
After setting $O=P=Q=R=S=T=0$ in (\ref{phi2}), we get the expression of a general closed and primitive 3-form
\begin{equation}
\varphi=A(e^{135}\!+\!e^{136})\!+\!C(e^{145}\!-\!e^{146})\!+\!E(e^{235}\!-\!e^{236})\!+\!G(e^{245}\!+\!e^{246})\!+\!(Me^5\!+\!Ne^6)(e^{12}\!-\!e^{34}).
\end{equation}

As a consequence, we conclude that
\begin{prop}\label{prop2b}~\\
\begin{enumerate}
\item $\textrm{H}^3(\frg;\rr)$ has a basis
\[e^5(e^{12}-e^{34}),e^6(e^{12}-e^{34}),e^5(e^{12}+e^{34}),e^6(e^{12}+e^{34}).\]
\item $\textrm{PH}^3(\frg;\rr)$ has a basis
\[e^5(e^{12}-e^{34}),e^6(e^{12}-e^{34}).\]
\item $\textrm{SH}_+^3(\frg;\rr)$ has a basis
\[e^5(e^{12}-e^{34}),e^6(e^{12}-e^{34}).\]
\item $\textrm{SH}_-^3(\frg;\rr)$ has a basis
\[e^5(e^{12}-e^{34}),e^6(e^{12}-e^{34}).\]
\end{enumerate}
\end{prop}

By above calculation, it is clear that the $\clf$-harmonicity condition translates to

\[\begin{cases}&A-B=C+D=E+F=G-H=I+O=J+P=K+Q=L+R=0,\\
&\hat A-\hat B=\hat C+\hat D=\hat E+\hat F=\hat G-\hat H=\hat I+\hat O=\hat J+\hat P=\hat K+\hat Q=\hat L+\hat R=0.\end{cases}\]

By applying formulae in Lemma \ref{f}, the $F$-harmonicity of $\varphi$ is reduced to the following system of polynomial equations
\begin{equation}\label{alg2}
\begin{cases}&A(4CE-(M-N)^2+(S-T)^2)=0,\\
&C(4AG+(M+N)^2-(S+T)^2)=0,\\
&E(4AG+(M+N)^2-(S+T)^2)=0,\\
&G(4CE-(M-N)^2+(S-T)^2)=0,\\
\end{cases}
\end{equation}
As $M,N,S,T$ are parameters for $\textrm{H}^3(\frg;\rr)$ and $A,C,E,G$ are parameters for exact forms, it follows that every cohomology class has an $\clf$-harmonic representative by setting $A=C=E=G=0$. On the other hand, there exists no cohomology class such that every representative is $\clf$-harmonic. Moreover, we never have the uniqueness. In summary, we have that
\[\varnothing=PL_\clf\subset L_\clf=\textrm{H}^3(\frg;\rr)\cong\rr^4.\]

The situation with primitive $\varphi$ is similar: we simply set $S=T=0$ in (\ref{alg2}) and $M,N$ are cohomology parameters while $A,C,E,G$ are exact form parameters. All three cohomologies $\textrm{PH}^3(\frg;\rr)$ and $\textrm{SH}^3_\pm(\frg;\rr)$ are all isomorphic and we have that \[\varnothing=PL_\clf\subset L_\clf=\textrm{PH}^3(\frg;\rr)=\textrm{SH}^3_\pm(\frg;\rr)\cong\rr^2.\]

\begin{prop}
If $\varphi$ is $\clf$-harmonic, then $\clq(\varphi)\geq 0$.
\end{prop}
\begin{proof}
If $\varphi$ is closed, by plugging (\ref{closed}) in Lemma \ref{q}, we get
\[\begin{split}&\frac{\clq(\varphi)}{4}\\ =&16ACEG+4CE((M+N)^2-(S+T)^2)-4AG((M-N)^2-(S-T)^2)+4(MT-NS)^2\\
=&\left[4AG+(M+N)^2-(S+T)^2\right]\left[4CE-(M-N)^2+(S-T)^2\right]+(M^2-N^2-S^2+T^2)^2.\end{split}\]
Suppose $\varphi$ is $\clf$-harmonic, then according to (\ref{alg2}), at least one of the following three cases must happen
\[4AG+(M+N)^2-(S+T)^2=0\textrm{ or }4CE-(M-N)^2+(S-T)^2=0\textrm{ or }A=C=E=G=0.\]
In the first two cases, we have that
\[\frac{\clq(\varphi)}{4}=(M^2-N^2-S^2+T^2)^2\geq 0.\]
In the third case, we have that
\[\frac{\clq(\varphi)}{4}=4(MT-NS)^2\geq 0.\]
\end{proof}

Now let us turn to the Type IIA flow on $(\frg,\omega)$ with general initial data. With $\varphi$ in (\ref{generalphi}), Proposition \ref{prop2a}(f) implies that
\[\begin{split}\frac{1}{2\lambda^2}\ud\Lambda_\omega\ud\clf(\varphi)=&(\hat A-\hat B)(e^{135}+e^{136})-(\hat C+\hat D)(e^{145}+e^{146})\\
&-(\hat E+\hat F)(e^{235}-e^{236})+(\hat G-\hat H)(e^{245}+e^{246}).\end{split}\]
Let
\bea
\alpha=\frac{A+B}{2},&&\alpha'=\frac{A-B}{2},\nonumber\\
\beta=\frac{C-D}{2},&&\beta'=\frac{C+D}{2},\nonumber\\
\gamma=\frac{E-F}{2},&&\gamma'=\frac{E+F}{2},\nonumber\\
\delta=-\frac{G+H}{2},&&\delta'=\frac{G-H}{2},\nonumber
\eea
then we can express $\varphi$ as
\bea
\varphi&=&\alpha(e^{135}+e^{136})+\alpha'(e^{135}-e^{136})+\beta(e^{145}-e^{146})+\beta'(e^{145}+e^{146})\nonumber\\
&&+\gamma(e^{235}-e^{236})+\gamma'(e^{235}+e^{236})-\delta(e^{245}+e^{246})+\delta'(e^{245}-e^{246})\nonumber\\
&&+(Ie^1+Je^2)(e^{34}-e^{56})+(Ke^3+Le^4)(e^{12}-e^{56})+(Me^5+Ne^6)(e^{12}-e^{34})\nonumber\\
&&+(Oe^1+Pe^2)(e^{34}+e^{56})+(Qe^3+Re^4)(e^{12}+e^{56})+(Se^5+Te^6)(e^{12}+e^{34}).\nonumber
\eea

It follows that the Type IIA flow reduces to the ODE system
\bea
\begin{cases}
\pt_t\alpha&=2\lambda^2(\hat A-\hat B),\\
\pt_t\beta&=-2\lambda^2(\hat C+\hat D),\\
\pt_t\gamma&=-2\lambda^2(\hat E+\hat F),\\
\pt_t\delta&=-2\lambda^2(\hat G-\hat H),
\end{cases}\nonumber
\eea
where $\alpha',\beta',\gamma',\delta',I,J,\dots,S,T$ are all constants. Expand the above system using Lemma \ref{generalf}, we get
\bea
\begin{cases}
\dfrac{\pt_t\alpha}{4\lambda^2}~=&4\alpha\beta\gamma-\alpha((M-N)^2-(S-T)^2)-2\beta(K^2-Q^2)-2\gamma(I^2-O^2)\\ &-2\alpha'(\alpha\delta'+\beta\gamma'+\gamma\beta'+\delta\alpha')+\clr_1,\\
\dfrac{\pt_t\beta}{4\lambda^2}~=&4\alpha\beta\delta-\beta((M+N)^2-(S+T)^2)-2\alpha(L^2-R^2)-2\delta(I^2-O^2)\\ &-2\beta'(-\alpha\delta'-\beta\gamma'+\gamma\beta'+\delta\alpha')+\clr_2,\\
\dfrac{\pt_t\gamma}{4\lambda^2}~=&4\alpha\gamma\delta-\gamma((M+N)^2-(S+T)^2)-2\alpha(J^2-P^2)-2\delta(K^2-Q^2)\\ &-2\gamma'(-\alpha\delta'+\beta\gamma'-\gamma\beta'+\delta\alpha')+\clr_3,\\
\dfrac{\pt_t\delta}{4\lambda^2}~=&4\beta\gamma\delta-\delta((M-N)^2-(S-T)^2)-2\beta(J^2-P^2)-2\gamma(L^2-R^2)\\ &-2\delta'(\alpha\delta'-\beta\gamma'-\gamma\beta'+\delta\alpha')+\clr_4,\\
\end{cases}\label{system}
\eea
where $\clr_1,\clr_2,\clr_3,\clr_4$ are constants.

The ODE system (\ref{system}) is rather complicated, and its long-time behavior depends highly on the choice of initial data. On one hand, $\clf$-harmonic 3-forms, as a special kind of stationary points of (\ref{system}), are long-time solutions. On the other hand, we shall prove that for the most desired initial data for the classical Type IIA flow introduced in \cite{fei2021b}, namely those primitive, closed, and positive 3-forms, the ODE system (\ref{system}) always develop a finite-time singularity.

As proved in \cite{fei2021b}, the primitiveness, closedness and positivity of 3-forms are all preserved under the Type IIA flow. In our setting, being closed and primitive is equivalent to
\[\alpha'=\beta'=\gamma'=\delta'=I=J=K=L=O=P=Q=R=S=T=0,\]
which further implies that
\[\clr_1=\clr_2=\clr_3=\clr_4=0.\]
Therefore for closed and primitive initial data, the Type IIA flow reduces to the following ODE system:
\bea
\begin{cases}\label{solvsystem}
\pt_t\alpha&=4\lambda^2\alpha(4\beta\gamma-(M-N)^2),\\
\pt_t\beta&=4\lambda^2\beta(4\alpha\delta-(M+N)^2),\\
\pt_t\gamma&=4\lambda^2\gamma(4\alpha\delta-(M+N)^2),\\
\pt_t\delta&=4\lambda^2\delta(4\beta\gamma-(M-N)^2),
\end{cases}
\eea
where $M,N$ are constants determined by the initial data.

Now let us turn to the ODE system (\ref{solvsystem}). We shall only consider the case where the initial data for $\varphi$ is positive, which is an open condition saying that the matrix
\[\begin{bmatrix}2\alpha\beta & & \alpha(N-M) & \beta(M+N) &  & \\
& 2\gamma\delta & \gamma(M+N) & \delta(M-N) & & \\
\alpha(N-M) & \gamma(M+N) & 2\alpha\gamma & & &\\
\beta(M+N) & \delta(M-N) & & 2\beta\delta & &\\
& & & & \alpha\delta+\beta\gamma-M^2 & \alpha\delta-\beta\gamma-MN\\
& & & & \alpha\delta-\beta\gamma-MN & \alpha\delta+\beta\gamma-N^2
\end{bmatrix}\]
is positive definite. Moreover, we know the positive definiteness of the above matrix is preserved under the Type IIA flow \cite{fei2021b}. By Sylvester's criterion, the positivity of the above matrix is equivalent to the following system of inequalities
\bea
\begin{cases}&\alpha,\beta,\gamma,\delta \textrm{ are all positive or all negative},\\
&\alpha\delta+\beta\gamma>M^2,\quad \alpha\delta+\beta\gamma>N^2,\\
&\dfrac{\clq(\varphi)}{16}=-4\alpha\beta\gamma\delta+\alpha\delta(M-N)^2+\beta\gamma(M+N)^2<0,
\end{cases}\label{initial}
\eea
which are all preserved under the Type IIA flow.

The behavior of the ODE system (\ref{solvsystem}) when $M=N=0$ is analyzed in \cite{fei2021b}. For the general case, we have:
\begin{prop}
For any value of $M$ and $N$, the ODE system (\ref{solvsystem}) has a finite-time singularity. Let $T<\infty$ be its maximal existence time, then the limit
\[\lim_{t\to T}\alpha(t)^{-1}\varphi(t)=\varphi_{\infty}\]
exists smoothly. Moreover, $\varphi_{\infty}$ defines a harmonic almost complex structure, as detailed in \cite{fei2021b}.
\end{prop}
\begin{proof}
It is obvious that $\alpha/\delta$ and $\beta/\gamma$ are positive constants along the flow.

Without loss of generality, we may assume that $\alpha$, $\beta$, $\gamma$, and $\delta$ are all positive, otherwise we may consider the evolution equation for $-\alpha$, $-\beta$, $-\gamma$, and $-\delta$ instead. Since (\ref{initial}) holds along the flow, we know that
\[4\alpha\beta\gamma\delta>\alpha\delta(M-N)^2,\quad 4\alpha\beta\gamma\delta>\beta\gamma(M+N)^2,\]
so the right hand side of (\ref{solvsystem}) are all positive. It follows that there exists a small positive number $c$ depending on the initial data such that $\p_tf\geq cf$ holds when $f=\alpha,\beta,\gamma$ or $\delta$, therefore all of $\alpha$, $\beta$, $\gamma$, and $\delta$ have at least exponential growth.

Now let $u=4\alpha\delta$ and $v=4\beta\gamma$. The pair $(u,v)$ satisfies
\bea
\begin{cases}
&\p_tu=2\lambda^2u(v-(M-N)^2),\\
&\p_tv=2\lambda^2v(u-(M+N)^2),
\end{cases}\label{uv}
\eea
together with
\bea
u,v>0\textrm{ and } uv>u(M-N)^2+v(M+N)^2.\label{mother}
\eea
In fact, the last two lines in the system of inequalities (\ref{initial}) can be derived from (\ref{mother}) as follows. First we deduce that $u>(M+N)^2$ and $v>(M-N)^2$. So we get
\[\left(\frac{u+v-2M^2-2N^2}{2}\right)^2\geq (u-(M+N)^2)(v-(M-N)^2)>(M^2-N^2)^2.\]
It follows that 
\[u+v>2M^2+2N^2+2|M^2-N^2|=4\max\{M^2,N^2\}.\]
Let $S=\max\{(M+N)^2,(M-N)^2\}$ and we shall compare the ODE system (\ref{uv}) with the following ODE system
\bea
\begin{cases}
&\p_tu=2\lambda^2u(v-S),\\
&\p_tv=2\lambda^2v(u-S).
\end{cases}\label{uvnormal}
\eea
By the ODE comparison theorem, if we can show that (\ref{uvnormal}) blows up in finite time, then (\ref{uv}) must blow up in even shorter time. It turns out that (\ref{uvnormal}) can be solved explicitly as follows. By taking the difference of the two equations in (\ref{uvnormal}), we get
\[\p_t(u-v)+2\lambda^2S(u-v)=0.\]
Therefore there exists a constant $C_0=u_0-v_0$ such that
\[u-v=C_0e^{-2\lambda^2St}.\]
Plug it back in (\ref{uvnormal}), we get
\bea
\p_t u=2\lambda^2u(u-C_0e^{-2\lambda^2St}-S).\label{u}
\eea
Let $w=e^{2\lambda^2St}u$, then from (\ref{u}) we get
\[\frac{\p_tw}{w(w-C_0)}=2\lambda^2e^{-2\lambda^2St}.\]
It follows that
\[w=\frac{C_0}{1-\frac{w_0-C_0}{w_0}e^{-\frac{C_0}{S}(e^{-2\lambda^2St}-1)}}.\]
As a consequence, we see that $w$ blows up at finite time $T'$, where
\[\begin{split}T'&=-\frac{1}{2\lambda^2S}\log\left[1+\frac{S}{C_0}\log\frac{w_0-C_0}{w_0}\right]\\
&=-\frac{1}{2\lambda^2S}\log\left[1+\frac{S}{\alpha_0\delta_0-\beta_0\gamma_0}\log\left(1 -\frac{\alpha_0\delta_0-\beta_0\gamma_0}{\alpha_0\delta_0}\right)\right].\end{split}\]
Now we can turn back to the ODE system (\ref{uv}), for which we now know that it blows up at a finite time $T$ and we have an estimate $T<T'$. Without loss of generality, we may assume $v(t)$ blows up at time $T$, i.e. $\lim_{t\to T}v(t)=+\infty$. Now we shall show that $u(t)$ also blows up at timer $T$.

By direct computation, we have
\[\p_t\left(\frac{(u-v)^2}{u^2}\right)=\frac{2\lambda^2}{u^2}\left[(u-v)((M+N)^2v-(M-N)^2u)-(u-v)^2(v-(M-N)^2)\right].\]
Firstly, it is easy to show that for any $\epsilon>0$, there exists a positive constant, whose optimal value is $C_\epsilon=\dfrac{8M^2N^2}{\epsilon}-(M+N)^2$, such that
\[(u-v)((M+N)^2v-(M-N)^2u)\leq C_\epsilon(u-v)^2+\epsilon u^2.\]
As $v(t)$ is increasing and $\lim_{t\to T}v(t)=+\infty$, there exists a time $T_1<T$ such that for any $T_1\leq t < T$, we have
\[v(t)-(M-N)^2\geq \frac{v(t)}{2}+C_\epsilon.\]
Therefore on the time interval $[T_1,T)$, we have
\bea
\p_t\left(\frac{(u-v)^2}{u^2}\right)\leq\frac{2\lambda^2}{u^2}\left[\epsilon u^2-\frac{v}{2}(u-v)^2\right]\leq 2\lambda^2\epsilon.\label{est}
\eea
Integrating it from $T_1$ to $t$ we see that
\[\left(\frac{u(t)-v(t)}{u(t)}\right)^2\leq C^2\]
on $[T_1,T)$, where the constant $C$ depends on $u(T_1)$, $v(T_1)$ and $\epsilon$. Therefore we know that on $[T_1,T)$, we have the estimate
\[v(t)\leq (1+C)u(t),\]
hence $\lim_{t\to\infty}u(t)=+\infty$.

Notice that the evolution equation of $u(t)$ can be written as
\[\frac{1}{2\lambda^2}\frac{\p_tu}{u}=v-(M-N)^2.\]
Integrating it, we get
\[\int_0^tv(s)\ud s=t(M-N)^2+\frac{1}{2\lambda^2}\log\frac{u(t)}{u(0)}\]
for any $0\leq t< T$. As we have proved that $\lim_{t\to T}u(t)=+\infty$, we can conclude
\[\lim_{t\to T}\int_0^tv(s)\ud s=+\infty.\]
Write $r=\dfrac{(u-v)^2}{u^2}$, then the first half of (\ref{est}) can be rewritten as
\[\p_tr+\lambda^2vr\leq 2\lambda^2\epsilon.\]
It follows that on $[T_1,T)$ we have
\[\p_t\left(e^{\lambda^2\int_0^tv(s)\ud s}r(t)\right)\leq 2\lambda^2\epsilon e^{\lambda^2\int_0^tv(s)\ud s}\leq \frac{2\epsilon}{v(T_1)}\lambda^2v(t)e^{\lambda^2\int_0^tv(s)\ud s}.\]
Integrating both sides, we get
\[e^{\lambda^2\int_0^tv(s)\ud s}r(t)\leq \frac{2\epsilon}{v(T_1)}\left(e^{\lambda^2\int_0^tv(s)\ud s}-e^{\lambda^2\int_0^{T_1}v(s)\ud s}\right)+e^{\lambda^2\int_0^{T_1}v(s)\ud s}r(T_1).\]
Divide both sides by $e^{\lambda^2\int_0^tv(s)\ud s}$ and let $t\to T$, we get
\[\lim_{t\to T}r(t)\leq\frac{2\epsilon}{v(T_1)}\]
for any $\epsilon>0$, hence we have proved
\bea
\lim_{t\to T}\frac{(u(t)-v(t))^2}{u^2(t)}=0,\label{est2}
\eea
or equivalently
\bea
\lim_{t\to T}\frac{u(t)}{v(t)}=1.\label{est3}
\eea
As $u=4\alpha\delta$ and $v=4\beta\gamma$ with $\alpha/\delta$ and $\beta/\gamma$ are positive constants, the limit (\ref{est3}) implies that
\[\begin{split}
&\lim_{t\to T}\frac{\beta(t)}{\alpha(t)}=\beta_\infty>0,\\
&\lim_{t\to T}\frac{\gamma(t)}{\alpha(t)}=\gamma_\infty>0,\\
&\lim_{t\to T}\frac{\delta(t)}{\alpha(t)}=\delta_\infty>0.
\end{split}\]
Let $\alpha_\infty=1$. Since $M$ and $N$ are constants, we conclude that
\[\lim_{t\to T}\frac{\varphi(t)}{\alpha(t)}=\alpha_\infty(e^{135}+e^{136})+\beta_\infty(e^{145}-e^{146})+\gamma_\infty(e^{235}-e^{236})-\delta_\infty(e^{245}+e^{246})\]
exists smoothly. Moreover, the estimate (\ref{est2}) implies that
\[\alpha_\infty\delta_\infty-\beta_\infty\gamma_\infty=0,\]
so from \cite{fei2021b} we know that the almost complex structure associated to $\varphi_\infty$ is harmonic.
\end{proof}
\begin{rmk}
The limit (\ref{est2}) can actually be strengthen to
\[\lim_{t\to T}\frac{(u(t)-v(t))^2}{u(t)}=0.\]
\end{rmk}    
 
\bibliographystyle{alpha}

\bibliography{C:/Users/benja/Dropbox/Documents/Source}

\end{document}